\documentclass[a4paper]{article}
\usepackage[american]{babel}

\usepackage{ifluatex}
\ifluatex
\usepackage{beautiful}
\else

\usepackage[utf8]{inputenc}
\usepackage[unicode]{hyperref}
\usepackage{amssymb,amsfonts,amsmath,amsthm}
\usepackage{graphicx}
\usepackage{txfonts}
\fi

\usepackage{tikz}
\usetikzlibrary{arrows,automata,matrix,positioning,decorations.pathreplacing}

\usepackage{newfloat}
\DeclareFloatingEnvironment[name=Algorithm]{algorithm}

\usepackage{biblatex}
\bibliography{document}

\newtheorem{theorem}{Theorem}[section]

\newtheorem{corollary}[theorem]{Corollary}

\newtheorem{lemma}[theorem]{Lemma}
\newtheorem{proposition}[theorem]{Proposition}

\def\E{\mathcal{E}}

\def\O{\mathcal{O}}

\def\CC{\mathbb{C}}

\def\FF{\mathbb{F}}

\def\NN{\mathbb{N}}
\def\PP{\mathbb{P}}
\def\QQ{\mathbb{Q}}

\def\ZZ{\mathbb{Z}}

\def\ppp{\mathfrak{p}}

\DeclareMathOperator{\End}{End}
\DeclareMathOperator{\GL}{GL}

\DeclareMathOperator{\cl}{cl}

\DeclareMathOperator{\ord}{ord}

\DeclareMathOperator{\tr}{tr}

\DeclareMathOperator{\numerator}{numerator}

\title{Iterative constructions of irreducible polynomials from isogenies}
\author{Alp Bassa \textsuperscript{a}
\and Gaetan Bisson \textsuperscript{b}
\and Roger Oyono \textsuperscript{b}
}
\date{\small
\textsuperscript{a} Department of Mathematics, Boğaziçi University, Turkey \\%
\textsuperscript{b} Laboratoire de mathématiques GAATI, University of French Polynesia
}

\begin{document}

\maketitle

\begin{abstract}
Let $S$ be a rational fraction and let $f$ be a polynomial over a finite field.
Consider the transform $T(f)=\operatorname{numerator}(f(S))$. In certain cases,
the polynomials $f$, $T(f)$, $T(T(f))\dots$ are all irreducible. For instance, in odd
characteristic, this is the case for the rational fraction $S=(x^2+1)/(2x)$,
known as the $R$-transform, and for a positive density of all irreducible
polynomials $f$.

We interpret these transforms in terms of isogenies of elliptic curves. Using
complex multiplication theory, we devise algorithms to generate a large number
of other rational fractions $S$, each of which yields infinite families of
irreducible polynomials for a positive density of starting irreducible
polynomials $f$.

\medskip

\end{abstract}


\section{Introduction}

Let $S\in\QQ(x)$ be a rational fraction. Let $\FF_q$ be a finite field where
the reduction of the denominator of $S$ does not vanish. For any polynomial
$f\in\FF_q[x]$ we define the $S$-transform of $f$ as the polynomial
$T_S(f)=\numerator\left(f(S(x))\right)$ and we let
\[
I_S(f)=\left(T_S^k(f)\right)_{k\geq 0}
\]
denote the family of polynomials obtained by applying
$T_S^k=T_S\circ\cdots\circ T_S$ (the composition of $k$ copies of $T_S$) to the
polynomial $f$. We say that $S$ induces an irreducible family from $f$ if the
polynomials in the family $I_S(f)$ are all irreducible.

For example, well-known transforms include the so-called $Q$-transform which
uses the rational fraction $Q(x)=\frac{x^2+1}{x}$ and the so-called $R$-transform which
uses the rational fraction $R(x)=\frac{1}{2}\frac{x^2+1}{x}$; more explicitly, we have
\begin{align*}
T_Q(f)(x) &= x^{\deg(f)}\cdot f\left(\frac{x^2+1}{x}\right), \\
T_R(f)(x) &= (2x)^{\deg(f)}\cdot f\left(\frac{1}{2}\frac{x^2+1}{x}\right).
\end{align*}
Those two transforms have been studied extensively and are known to induce
irreducible families.

\begin{theorem}[Q-transform \cite{varshamov-recurrent, meyn-qtransform, kyuregyan-gf2s}]
Let $q=2^r$ and let $f(x)=\sum_{i=0}^{n} a_i x^i$ be an irreducible polynomial
in $\FF_q[x]$ with $a_n=1$. Let $\tr$ denote the trace from $\FF_q$ to $\FF_2$.
Assuming $\tr(a_{n-1})=\tr(a_1/a_0)=1$, the fraction $Q$ induces an irreducible
family from $f$.
\end{theorem}

\begin{theorem}[R-transform \cite{cohen-rtransform}]
Let $q$ be an odd prime power and let $f$ be a monic irreducible polynomial in
$\FF_q[x]$. Assume that $f(1)f(-1)$ is not a square in $\FF_q$ and, if
$q=3\bmod 4$, assume additionally that $\deg(f)$ is even. The fraction $R$
induces an irreducible family from $f$.
\end{theorem}

Recently there has been interest in constructing transforms $T$ which induce
irreducible families. We note the work of Bassa and Menares using
Galois theory on function fields \cite{bassa-menares-g} and 
using multiplicative group theory \cite{bassa-menares-r}.

In this article we construct such transforms from isogenies of elliptic curves.
Our main results are algorithms which generate a large diversity of transforms.


\section{General framework}

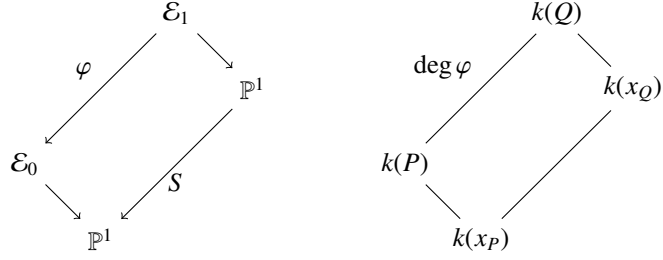
\begin{figure}
\begin{center}
\begin{tikzpicture}
    \node at (0,0) (E0){$\E_0$};
    \node at (2,2) (E1){$\E_1$};
    \node at (1,-1) (P0){$\PP^1$};
    \node at (3,1) (P1){$\PP^1$};
    \draw[->] (E0) -- (P0);
    \draw[->] (E1) -- (P1);
    \draw[<-] (E0) -- (E1) node[midway,above left] {$\varphi$};
    \draw[<-] (P0) -- (P1) node[midway,below] {$S$};

    \node at (5,0) (kP){$k(P)$};
    \node at (7,2) (kQ){$k(Q)$};
    \node at (6,-1) (kPx){$k(x_P)$};
    \node at (8,1) (kQx){$k(x_Q)$};
    \draw (kPx) -- (kP);
    \draw (kPx) -- (kQx);
    \draw (kP) -- (kQ) node[midway,above left] {$\deg\varphi$};
    \draw (kQx) -- (kQ);
\end{tikzpicture}
\end{center}
\caption{
On the left, the projection of an isogeny $\E_0\stackrel{\varphi}{\longleftarrow}\E_1$ to its Lattès map $S$;
on the right, the corresponding field extensions where the points satisfy $P=\varphi(Q)$.}
\label{fig:com-ext}
\end{figure}

We first explain the relationship between the transform $T_S$ and isogenies.
Let $\E_0\stackrel{\varphi}{\longleftarrow}\E_1$ be an isogeny of elliptic
curves in Weierstrass form defined over a finite field $k$. Consider two points
$P\in\E_0(\overline k)$ and $Q\in\E_1(\overline k)$ satisfying $P=\varphi(Q)$
and such that $[k(Q):k(P)]=\deg\varphi$. Since $\varphi$ commutes with the
involution endomorphism $(x,y)\mapsto (x,-y)$, quotienting out by it yields the
commutative diagram on the left of Figure~\ref{fig:com-ext} where the arrows to
the projective line are the projections of points on their $x$-coordinate and
where $S$ denotes the $x$-coordinate map of the isogeny $\varphi$.

This induces the field extensions diagram on the right of
Figure~\ref{fig:com-ext}. If either $\deg\varphi$ is odd or $[k(P):k(x_P)]=2$,
this implies $[k(x_Q):k(x_P)]=\deg\varphi$. Consequently, denoting by $f(x)\in
k[x]$ the minimal polynomial of $x_P$ over $k$, since $f(S(x_Q))=f(x_P)=0$, the
minimal polynomial of $x_Q$ over $k$ is $T_S(f(x))$ and the latter is therefore
irreducible.

To iterate this construction, we require a criteria on the isogeny $\varphi$
which ensures that the condition $[k(Q):k(P)]=\deg\varphi$ holds under further
compositions by $\varphi$. We begin with a simple but key lemma which describes
the action of the Frobenius endomorphism in explicit terms.

\begin{lemma}
Let $\E_0$ and $\E_1$ be elliptic curves and
$\E_0\stackrel{\varphi}{\longleftarrow}\E_1$ be a separable isogeny defined
over a finite field $k$. Fix a point $P\in\E_0(\overline k)$ and denote by
$\pi$ the $k(P)$-Frobenius endomorphism on $\E_1$. If all points in the kernel
of $\varphi$ are defined over $k(P)$, then there exists a point
$F\in\ker\varphi$ such that, for all points $Q\in \varphi^{-1}(P)$ we have
$\pi^n(Q)=Q+nF$ for all $n\in\NN$.
\end{lemma}

\begin{proof}
Consider a point $Q\in\varphi^{-1}(P)$. The inverse image of $P$ by $\varphi$
can then be written as $\varphi^{-1}(P)=\{Q+R : R\in\ker\varphi\}$; in
particular, all points in the fiber have the same field of definition. Since
$\varphi(\pi(Q))=\pi(\varphi(Q))=\pi(P)=P$, there exists $F\in\ker \varphi$
such that $\pi(Q)=Q+F$. All other points $Q'\in \varphi^{-1}(P)$, being of the
form $Q'=Q+R$, also satisfy $\pi(Q')=\pi(Q+R)=\pi(Q)+\pi(R)=Q+F+R=Q'+F$.
Finally, since $F\in\E_1(k(P))$, we obtain
\[
\pi^n(Q)=\pi^{n-1}(Q+F)=\pi^{n-1}(Q)+\pi(F)=\pi^{n-1}(Q)+F=\cdots=Q+nF.
\]
\end{proof}

We deduce the theorem below which gives precisely the criteria we required.

\begin{theorem}
\label{th:main}
Let
$\E_0\stackrel{\varphi_0}{\longleftarrow}\E_1\stackrel{\varphi_1}{\longleftarrow}\E_2$
be two separable isogenies of respective degree $\ell_0$ and $\ell_1$ defined
over a finite field $k$. Suppose that all prime factors of $\ell_1$ divide
$\ell_0$. Fix a point $P\in\E_0(\overline k)$ and assume that the kernel
$\ker(\varphi_0\circ\varphi_1)$ is cyclic and that all its points are
$k(P)$-rational. Then, all points $Q\in(\varphi_0\circ\varphi_1)^{-1}(P)$
satisfying $\left[k(\varphi_1(Q)):k(P)\right]=\ell_0$ also satisfy
$\left[k(Q):k(P)\right]=\ell_0\ell_1$.
\end{theorem}

\begin{proof}
Since $G=\ker(\varphi_0\circ\varphi_1)$ is a cyclic subgroup of $\E_2(k(P))$ of
order $\ell_0\ell_1$, it admits a unique subgroup of order $\ell_1$, namely
$\ell_0\cdot G$, which by uniqueness is equal to $\ker\varphi_1$. Denote by
$\pi$ the $k(P)$-Frobenius endomorphism on $\E_2$. By the lemma, there exists a
point $F\in G$ satisfying $\pi^n(Q)=Q+nF$. In particular, its order is
$\ord(F)=[k(Q):k(P)]$. Since
$\pi(\varphi_1(Q))=\varphi_1(\pi(Q))=\varphi_1(Q+F)=\varphi_1(Q)+\varphi_1(F)$,
we similarly have $[k(\varphi_1(Q)):k(P)]=\ord(\varphi_1(F))$.

Assume now $[k(\varphi_1(Q)):k(P)]=\ell_0$, that is,
$\ord(\varphi_1(F))=\ell_0$. We claim $\ord(F)=\ell_0\ell_1$. Suppose otherwise
that $\ord(F)<\ell_0\ell_1$. Then, we can write $F=p\cdot T$ for some $T\in G$
and some prime $p$ dividing $\ell_0\ell_1$. As all prime divisors of $\ell_1$
are divisors of $\ell_0$, we have $p\mid\ell_0$. This implies $\ell_0/p\cdot
F=\ell_0\cdot T\in {\ell_0}\cdot G=\ker \varphi_1$, that is, $\ell_0/p\cdot
\varphi_1(F)={\mathcal O}_{\E_1}$, which contradicts $\ord
(\varphi(F))=\ell_0$. We thus obtain $\ord (F)=\ell_0 \ell_1=
\left[k(Q):k(P)\right]$ as claimed.
\end{proof}

Note that the simplest setting where this result can be iterated is when
$\E_0=\E_1=\E_2$ and the endomorphisms $\varphi_0$ and $\varphi_1$ are
identical. This yields the following corollary where we assume that
$\deg\varphi$ is odd for simplicity.

\begin{corollary}
\label{cor:iter}
Let $\E$ be an elliptic curve, $\varphi:\E\to\E$ a separable endomorphism of
odd degree defined over a finite field $k$, and $P\in\E(\overline k)$ a point.
Suppose that the subgroup $\ker(\varphi\circ\varphi)$ is cyclic and that all its points
are $k(P)$-rational. Denote by $S$ the $x$-coordinate map of $\varphi$ and by
$f$ the minimal polynomial of $x_P$ over $k$. Then, if $T_S(f)$ is
irreducible, so are all polynomials in the family $I_S(f)$.
\end{corollary}

The map $S$ is what is known as a Lattès map \cite{lattes-maps}: it is the
projection of an endomorphism $\varphi:\E\to\E$ through a finite separable
cover $\E\to\PP^1$ (in this case, the projection on the $x$-coordinate).

Note that the condition on $\ker(\varphi\circ\varphi)$ being cyclic is
equivalent to no isogeny factor of $\varphi$ being dual to another. In the
particular case where $\varphi$ has prime degree, it reduces to $\varphi$ not
being its own dual.



\subsection{Möbius transforms}

For any matrix $m\in\GL_2(\ZZ)$, define the rational fraction
\[
M_m(x)=\frac{\alpha x+\beta}{\gamma x+\delta},
\qquad
\text{where~}
m=\begin{pmatrix}\alpha & \beta \\ \gamma & \delta\end{pmatrix}
.
\]
If $S$ is a rational fraction in $\QQ(x)$, we define the corresponding Möbius
transform of $S$ as the composition $S'=M_{m^{-1}}\circ S\circ M_m$. Note that
the fraction $S$ induces an irreducible family from a given polynomial $f$ if
and only if $S'$ does. Thus we may apply Möbius transforms to any rational
fraction while preserving its ability to induce irreducible families, for
instance in order to try and reduce the size of its coefficients.

Our efforts will from now on be focused on finding isogenies $\varphi:\E\to\E$
which satisfy the conditions of Corollary~\ref{cor:iter} and obtaining the
corresponding rational fractions $S$; we will purposely not look for associated
points $P$ and polynomials $f$. Nevertheless, in Section~\ref{sec:density}, we
will compute for each selected rational fraction $S$, the density of
irreducible polynomials of a given degree in a given finite field for which $S$
induces irreducible families.


\section{The Verschiebung endomorphism}

Let $\varphi:\E\to\E$ be a separable endomorphism of prime degree $\ell$
defined over a finite field $\FF_q$. In this section we consider the case where
$\ell$ divides $q$. Since the multiplication-by-$q$ map satisfies
$[q]=\pi\widehat\pi$, the endomorphism is either the Frobenius $\pi$, which is
purely inseparable, or its dual, the Verschiebung $\widehat\pi$, which is
separable if and only if the elliptic curve $\E$ is ordinary.

We may thus specialize Corollary~\ref{cor:iter} to the case where $q$ is an odd
prime and $\varphi=\widehat\pi$.

\begin{proposition}
Let $\varphi_0:E_1\to E_0$ be a separable isogeny of odd degree $\ell_0\neq p$
defined over a finite field $\FF_p$ with $p\neq 2$. Suppose the subgroup
$\ker(\varphi_0\circ\widehat\pi)$ is cyclic and all its points are rational.
Denote by $S$ the $x$-coordinate map of the $\widehat\pi$ and by $f$ the kernel
polynomial of $\varphi_0$. Then, if $f$ is irreducible, $S$ induces an
irreducible family from $f$.
\end{proposition}

In order to compute the $x$-coordinate of the Verschiebung endomorphism on an
elliptic curve $\E$, we use Algorithm~\ref{alg:verschiebung}.

\begin{algorithm}
\begin{center}
\begin{tabular}{rl}
\textsc{Input:} & An elliptic curve $\E$ defined over a finite field $\FF_q$. \\
\textsc{Output:} & The $x$-coordinate map of the Verschiebung endomorphism. \medskip \\
1. & Compute the division polynomial $\varphi_q(x)$ \\
   & for the multiplication-by-$q$ map on $\E$. \\
2. & Return $\varphi_q(x^{1/q})$.
\end{tabular}
\end{center}
\caption{Computing the Lattès map of the Verschiebung endomorphism of an elliptic curve defined over a finite field.}
\label{alg:verschiebung}
\end{algorithm}


We have computed all the rational fractions obtained using this algorithm,
including by composing with the Möbius map. Table~\ref{tbl:verschiebung} gives,
for small powers of two (even though Theorem~\ref{th:main} does not apply to
them, we find they still induce irreducible families; see
section~\ref{sec:density}) and small odd primes $q$, the number $N$ of such
transforms, and a representative element selected for having lowest Hamming
weight.

\begin{table}
\begin{center}
\begin{tabular}{r|r|l}
$q$ & $N$ & \textsc{representative fraction} \\
\hline
\hline
$2$ & $6$ & $x/(x^2 + 1)$ \\
$4$ & $180$ & $(x^4 + x^2 + 1)/(x^3 + x)$ \\
$8$ & $3528$ & $(x^7 + x)/(x^8 + x^6 + x^4 + x^2 + 1)$ \\
\hline
$3$ & $36$ & $(x^3 + x^2 + x + 2)/x^2$ \\
$5$ & $345$ & $(2x^5 + x)/(x^4 + 2)$ \\
$7$ & $1428$ & $(5x^7 + x^4 + 6x)/(x^6 + x^3 + 3)$ \\
$11$ & $8250$ & $(8x^{11} + x^9 + 7x^7 + 4x^3 + 10x)/(x^{10} + x^8 + 2x^4 + 7x^2 + 8)$ \\
\end{tabular}
\end{center}
\caption{Some rational fractions which induce irreducible families, computed as Lattès maps of Verschiebung endomorphisms.}
\label{tbl:verschiebung}
\end{table}

Note that for $q=2$ this method yields the well-known $Q$-transform.


\section{Isogenies of ordinary curves over finite fields}

Let $\E$ be an ordinary elliptic curve defined over a finite field $\FF_q$ and
denote by $\pi$ its Frobenius endomorphism. Its endomorphism ring $\End(\E)$ is
an order in the imaginary quadratic field $K=\QQ(\pi)$ containing $\ZZ[\pi]$.
Isogenies $\varphi:\E\to\E'$ of prime degree $\ell\nmid q$ fall into one of
two categories:
\begin{enumerate}
\item So-called \emph{horizontal} isogenies satisfy $\End(\E)=\End(\E')$ and
	are described by the theory of complex multiplication
	\cite{shimura-taniyama} which states that the ideal class group
	$\cl(\O)$ acts faithfully and transitively on the set of isomorphism
	classes of elliptic curves $\E$ satisfying $\End(\E)\simeq\O$.
\item Other prime-degree isogenies are said to be \emph{vertical} and display
	the so-called volcano structure \cite{kohel-phd,fouquet-morain}.
\end{enumerate}

\begin{figure}
\begin{center}
\begin{tikzpicture}[nodes={circle, fill, scale=0.7}, thick]
\node[fill=none,circle=none] (A) {};
\node (B0) [left of=A] {};
\node (B1) [above of=A] {};
\node (B2) [right of=A] {};
\node (B3) [below of=A] {};
\node (C0) [left of=B0] {};
\node (C1) [above of=B1] {};
\node (C2) [right of=B2] {};
\node (C3) [below of=B3] {};
\node[fill=none,circle=none] (D0) [left of=C0] {};
\node[fill=none,circle=none] (D1) [above of=C1] {};
\node[fill=none,circle=none] (D2) [right of=C2] {};
\node[fill=none,circle=none] (D3) [below of=C3] {};
\node (E0) [below of=D0] {};
\node (E1) [above of=D0] {};
\node (E2) [left of=D1] {};
\node (E3) [right of=D1] {};
\node (E4) [above of=D2] {};
\node (E5) [below of=D2] {};
\node (E6) [right of=D3] {};
\node (E7) [left of=D3] {};
\node (F0) [below of=E0] {};
\node (F1) [left of=E0] {};
\node (F2) [left of=E1] {};
\node (F3) [above of=E1] {};
\node (F4) [left of=E2] {};
\node (F5) [above of=E2] {};
\node (F6) [above of=E3] {};
\node (F7) [right of=E3] {};
\node (F8) [above of=E4] {};
\node (F9) [right of=E4] {};
\node (F10) [right of=E5] {};
\node (F11) [below of=E5] {};
\node (F12) [right of=E6] {};
\node (F13) [below of=E6] {};
\node (F14) [below of=E7] {};
\node (F15) [left of=E7] {};
\draw (B0) -- (B1) -- (B2) -- (B3) -- (B0);
\draw (B0) -- (C0);
\draw (B1) -- (C1);
\draw (B2) -- (C2);
\draw (B3) -- (C3);
\draw (C0) -- (E0);
\draw (C0) -- (E1);
\draw (C1) -- (E2);
\draw (C1) -- (E3);
\draw (C2) -- (E4);
\draw (C2) -- (E5);
\draw (C3) -- (E6);
\draw (C3) -- (E7);
\draw (E0) -- (F0);
\draw (E0) -- (F1);
\draw (E1) -- (F2);
\draw (E1) -- (F3);
\draw (E2) -- (F4);
\draw (E2) -- (F5);
\draw (E3) -- (F6);
\draw (E3) -- (F7);
\draw (E4) -- (F8);
\draw (E4) -- (F9);
\draw (E5) -- (F10);
\draw (E5) -- (F11);
\draw (E6) -- (F12);
\draw (E6) -- (F13);
\draw (E7) -- (F14);
\draw (E7) -- (F15);
\end{tikzpicture}
\end{center}
\caption{A connected component of a degree-$2$ isogeny graph displaying the so-called volcano structure;
the order in the class group of both primes of norm two is four (the length of the rim) and
the conductor $[\O_K:\ZZ[\pi]]$ has valuation three at two (the height of the trees).}
\label{fig:volcano}
\end{figure}
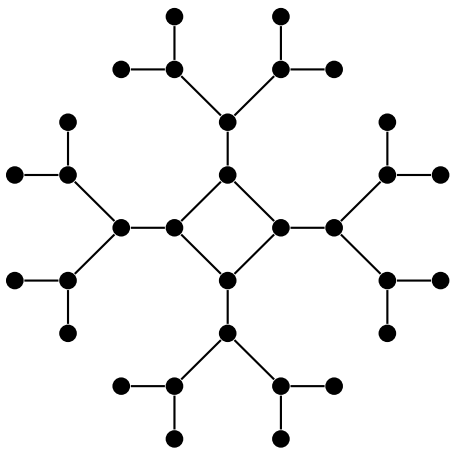

Connected components of degree-$\ell$ isogeny graphs thus have the shape
illustrated by Figure~\ref{fig:volcano}: elliptic curves with locally maximal
endomorphism ring are connected by horizontal isogenies which form a cycle (the
rim of the volcano) of length the order in the class group of an ideal of norm
$\ell$; other elliptic curves are located on trees formed of vertical isogenies
hanging from maximal curves; the graph is regular of degree $\ell+1$ except at
the leaves.

The goal of this section is to exploit this structure in order to construct
endomorphisms satisfying Corollary~\ref{cor:iter}.

\subsection{Prime isogenies of order one}

We take an ordinary elliptic curve $\E$ over a finite field $\FF_q$. We look
for a rational, cyclic endomorphism of $\E$ which is not its own dual. To
achieve this, we look for an isogeny of prime degree $\ell$ which splits into
primes of order one in the class group of $\End(\E)$. By complex multiplication
theory, such an isogeny maps $\E$ to an isomorphic curve.

Concretely, given a prime power $q$ and a prime $\ell$, we look for small
discriminants $\Delta$ for which $\ell$ splits into primes of order one in the
class group of $\QQ(\sqrt\Delta)$ and then use the Hilbert class polynomial
$H_{\Delta}$ to generate elliptic curves over $\FF_q$ with endomorphism algebra
$\QQ(\sqrt\Delta)$. We then compute the corresponding degree-$\ell$ isogeny and
extract its $x$-coordinate.

This yields Table~\ref{tbl:princip} where, as before, we select the lowest
Hamming-weight representative for each rational fraction $S$ under the action
by Möbius transforms.

\begin{table}
\begin{center}
\begin{tabular}{r|r|l}
$q$ & $\ell$ & \textsc{representative fraction} \\
\hline
\hline
$2$ & $3$ & $(x^3 + 1)/x^2$ \\
$5$ & $3$ & $x/(x^3 + x^2 + 1)$ \\
$7$ & $5$ & $(x^5 + x^4 + x^3 + 6x^2 + x)/(x^4 + x^3 + 4x^2 + x + 1)$ \\
$11$ & $2$ & $x/(x^2 + 1)$ \\
$11$ & $5$ & $(x^5 + 9x^4 + 10x^3 + 4x + 1)/(x^5 + x^3 + 9)$ \\
$17$ & $5$ & $(15x^5 + 3x^3 + x)/(x^5 + 3x^4 + 15x^3 + x^2 + 1)$ \\
\end{tabular}
\end{center}
\caption{Some rational fractions which induce irreducible families, computed as cyclic endomorphisms of prime degree.}
\label{tbl:princip}
\end{table}


\subsection{Multiple prime isogenies}

Let $\E$ be an ordinary elliptic curve over a finite field $\FF_q$. An
endomorphism of $\E$ may be constructed as the composition of multiple
horizontal isogenies forming a cycle in the isogeny graph. Equivalently, one
may search for products of prime ideals which are principal in the class groupe
of $\End(\E)$ and then construct the corresponding isogeny cycle through the
theory complex multiplication.

Here, we simply search for such endomorphisms, select those with cyclic kernel
and small degree, and apply Möbius transforms to reduce the Hamming weight of
the rational fraction describing their action on the $x$-coordinate. Among
others, we find the rational fractions of Table~\ref{tbl:multprim}.
Note that, as expected, the number of such endomorphisms grows with $q$.

\begin{table}
\begin{center}
\begin{tabular}{r|l}
$q$ & \textsc{representative fraction} \\
\hline
\hline
$2$ & $(x^7 + x^3 + x)/(x^9 + x^6 + x^5 + x^2 + 1)$ \\
$7$ & $(2x^{10} + 4x^9 + x^6 + x^5 + 3x^4 + 2x + 3)/(x^9 + 6x^8 + 5x^5 + x^2 + 4x)$ \\
$17$ & $(13x^9 + 4x^7 + x^5 + 8x)/(x^8 + 13x^4 + 11x^2 + 15)$ \\
$17$ & $(9x^9 + 3x^7 + 13x^5 + 10x^3 + 9x)/(x^8 + x^6 + x^4 + 4x^2 + 4)$ \\
$19$ & $(11x^4 + 17x^2 + 8)/(x^4 + 1)$ \\
$19$ & $(18x^9 + x^7 + 14x^5 + 11x^3 + 12x)/(x^{10} + 5x^8 + x^6 + 7x^4 + 5x^2 + 11)$ \\
$19$ & $(13x^5 + 10x^3 + 10x)/(x^6 + 1)$ \\
$19$ & $(11x^3 + 11x)/(x^4 + 1)$ \\
$19$ & $(16x^3 + x)/(x^4 + 4x^2 + 17)$ \\
$19$ & $(16x^{10} + 13x^6 + 12x^4 + 1)/(x^9 + x^7 + 10x^5 + 4x^3 + 16x)$ \\
$19$ & $(8x^6 + 14x^4 + 14x^2 + 8)/(x^5 + x^3 + x)$ \\
$19$ & $(x^4 + 7)/(x^4 + 14x^2 + 12)$
\end{tabular}
\end{center}
\caption{Some rational fractions which induce irreducible families, computed as cyclic endomorphisms of small degree.}
\label{tbl:multprim}
\end{table}

\section{Isogenies of ordinary curves over number fields}

Let $\E$ be an elliptic curve defined over a number field $K$ which admits a
rational endomorphism $\alpha:\E\to\E$ with cyclic kernel. For all places
$\ppp$ of good reduction where the localization of $\alpha$ still has cyclic
kernel, the reduction of $\alpha$ to $K/\ppp$ yields an endomorphism
$\varphi_1$ to which Theorem~\ref{th:main} may be applied. By the Cebotarev
density theorem, the rational fraction defining $\alpha$ in characteristic zero
can thus be applied to a positive density of finite fields.


\paragraph{Endomorphisms of degree two.}
The simplest case concerns elliptic curves defined over the rationals and
endowed with an endomorphism of degree two. Their $j$-invariants are the roots
of the modular polynomial $\Phi_2(j,j)$ and their endomorphisms can be computed
explicitly, resulting in the following theorem. See, for instance,
\cite[Proposition~2.3.1]{silverman-advanced}.

\begin{proposition}
There are exactly three isomorphism classes of elliptic curves over $\CC$ which
possess an endomorphism of degree 2. The following are representatives for
these curves and endomorphisms.
\begin{center}
\everymath{\displaystyle}\rm
\begin{tabular}{rlll}
(i)   & $E:y^2=x^3+x$, & $j=1728$, & $\alpha=1+\sqrt{-1}$, \\
      & \multicolumn{3}{l}{$[\alpha](x,y)=\left(\alpha^{-2}\left(x+\frac{1}{x}\right),\alpha^{-3}y\left(1-\frac{1}{x^2}\right)\right)$;} \smallskip\\
(ii)  & $E:y^2=x^3+4x^2+2x$, & $j=8000$, & $\alpha=\sqrt{-2}$, \\
      & \multicolumn{3}{l}{$[\alpha](x,y)=\left(\alpha^{-2}\left(x+4+\frac{2}{x}\right),\alpha^{-3}y\left(1-\frac{2}{x^2}\right)\right)$;} \smallskip\\
(iii) & $E:y^2=x^3-35x+98$, & $j=-3375$, & $\alpha=\frac{1+\sqrt{-7}}{2}$, \\
      & \multicolumn{3}{l}{$[\alpha](x,y)=\left(\alpha^{-2}\left(x-\frac{7(1-\alpha)^4}{x+\alpha^2-2}\right),\alpha^{-3}y\left(1+\frac{7(1-\alpha)^4}{(x+\alpha^2-2)^2}\right)\right)$.} \\
\end{tabular}
\end{center}
\end{proposition}

We note that the first endomorphism corresponds to the well-known $Q$-transform.

\paragraph{Endomorphisms of degree three.}
The same approach applies to higher-degree endomorphisms although the explicit
formulas describing them are much heavier that in the above degree-two case.

Consider for instance the elliptic curve $E:y^2 + 6xy + 4y = x^3$ with
$j$-invariant $54000$. Since it is a root of the modular polynomial
$\Phi_3(j,j)$, it admits a degree-three endomorphism. Indeed, this endomorphism
can be written explicitly as $\varphi \circ \phi$ where
$\alpha=\frac{1+\sqrt{-3}}{2}$ and
\begin{align*}
\phi(x,y)&= \left( x+\frac{24}{x}+\frac{16}{x^2},\,  y -\frac{64}{x^3} -\frac{24(6x+y+4)}{x^2} \right), \\
\varphi(x,y)&= \left(-\frac{1}{3}x - 4,\, -\frac{1}{3\sqrt{-3}} y +\frac{3-\sqrt{-3}}{3}x -\frac{2}{3\sqrt{-3}}+10 \right).
\end{align*}




\section{Density of irreducible families}
\label{sec:density}

Let $S$ be a rational fraction over a fixed finite field $\FF_q$. We are
interested in computing the density of irreducible polynomials $f$ of small
degree $d$ from which $S$ induces irreducible families. Through the Cebotarev
density theorem, the conditions Corollary~\ref{cor:iter} may be used to compute
these densities asymptotically. However this is burdensome and thus most
entries in the tables below were obtained through exhaustive computations.

First consider the rational fraction $S=(x^2+1)/x$ over $\FF_3$.
Table~\ref{tbl:niter} indicates, for selected integers $i$ and $d$, the density
of irreducible polynomials of degree $d$ which remain irreducible under only
just $i$ iterations of the transform $S$. Each column adds up to one.


\begin{table}
\begin{center}
\begin{tabular}{l|ccccc}
            &  $d=2$  &  $d=3$  &  $d=4$  &  $d=5$  &  $d=6$    \\
\hline                                                          
$i=0$       &  $1/3$  &  $1/2$  &  $4/9$  &  $1/2$  &  $14/29$  \\
$i=1$       &  $0$    &  $1/2$  &  $0$    &  $1/2$  &  $0$      \\
$i=\infty$  &  $2/3$  &  $0$    &  $5/9$  &  $0$    &  $15/29$  \\
\end{tabular}
\end{center}
\caption{Density of irreducible polynomials of degree $d$ over $\FF_3$ which
remain irreducible under only just $i$ iterations of the map $T_S$ where $S=(x^2+1)/x$.}
\label{tbl:niter}
\end{table}

In Table~\ref{tbl:density}, we only give the density of irreducible polynomials
of degree $d$ over $\FF_q$ from which the rational fraction $S$ induces
irreducible families. In the particular case where $S=(x^2+1)/x$ and $q=3$,
this corresponds to the line $i=\infty$ of Table~\ref{tbl:niter}.


\begin{table}
\begin{center}
\begin{tabular}{l|ccccc}
\multicolumn{6}{c}{$\displaystyle S=\frac{x^2+1}{x}$}
\smallskip \\
        &  $d=2$   &  $d=3$           &  $d=4$           &  $d=5$           &  $d=6$           \\
\hline                                                                                         
$q=2$   &  $1$     &  $0$             &  $1/3$           &  $1/3$           &  $2/9$           \\
$q=3$   &  $2/3$   &  $0$             &  $5/9$           &  $0$             &  $15/29$         \\
$q=5$   &  $0$     &  $0$             &  $0$             &  $0$             &  $0$             \\
$q=7$   &  $8/21$  &  $0$             &  $12/49$         &  $0$             &  $\approx 0.25$  \\
$q=11$  &  $8/55$  &  $\approx 0.12$  &  $\approx 0.13$  &  $\approx 0.12$  &  $\approx 0.12$  \\
$q=13$  &  $2/13$  &  $11/91$         &  $\approx 0.13$  &  $\approx 0.13$  &  $\approx 0.13$  \\
\end{tabular}
\end{center}

\begin{center}
\begin{tabular}{l|ccccc}
\multicolumn{6}{c}{$\displaystyle S=\frac{1}{2}\frac{x^2+1}{x}$}
\smallskip \\
        &  $d=2$   &  $d=3$  &  $d=4$           &  $d=5$  &  $d=6$           \\
\hline                                                                       
$q=3$   &  $2/3$   &  $0$    &  $5/9$           &  $0$    &  $15/29$         \\
$q=5$   &  $3/5$   &  $1/2$  &  $13/25$         &  $1/2$  &  $\approx 0.50$  \\
$q=7$   &  $4/7$   &  $0$    &  $25/49$         &  $0$    &  $\approx 0.50$  \\
$q=11$  &  $6/11$  &  $0$    &  $\approx 0.50$  &  $0$    &  $\approx 0.50$  \\
$q=13$  &  $7/13$  &  $1/2$  &  $\approx 0.50$  &  $1/2$  &  $\approx 0.50$  \\
$q=17$  &  $9/17$  &  $1/2$  &  $\approx 0.50$  &  $1/2$  &  $\approx 0.50$  \\
\end{tabular}
\end{center}

\begin{center}
\begin{tabular}{l|ccccc}
\multicolumn{6}{c}{$\displaystyle S=\alpha^{-2}\left(x-\frac{7(1-\alpha)^4}{x+\alpha^2-2}\right)$ where $\displaystyle\alpha=\frac{1+\sqrt{-7}}{2}$}
\smallskip \\
        &  $d=2$           &  $d=3$           &  $d=4$           &  $d=5$           &  $d=6$           \\  
\hline                                                                                                     
$q=11$  &  $16/55$         &  $13/55$         &  $\approx 0.26$  &  $\approx 0.25$  &  $\approx 0.25$  \\  
$q=23$  &  $\approx 0.25$  &  $\approx 0.25$  &  $\approx 0.25$  &  $\approx 0.25$  &  $\approx 0.25$  \\  
$q=29$  &  $8/29$          &  $\approx 0.25$  &  $\approx 0.25$  &  $\approx 0.25$  &  $\approx 0.25$  \\  
$q=37$  &  $\approx 0.26$  &  $\approx 0.25$  &  $\approx 0.25$  &  $\approx 0.25$  &  $\approx 0.25$  \\  
$q=43$  &  $\approx 0.25$  &  $\approx 0.25$  &  $\approx 0.25$  &  $\approx 0.25$  &  $\approx 0.25$  \\  
$q=53$  &  $\approx 0.26$  &  $\approx 0.25$  &  $\approx 0.25$  &  $\approx 0.25$  &  $\approx 0.25$  \\  
\end{tabular}
\end{center}

\caption{Density of irreducible polynomials of degree $d$ over $\FF_q$ from
which the rational fraction $S$ induces irreducible families.}
\label{tbl:density}
\end{table}

\section*{Acknowledgments}

Gaetan Bisson was supported by \emph{Agence Nationale de la Recherche} under the MELODIA project, grant number ANR-20-CE40-0013.


\printbibliography

\end{document}